\documentclass[preprint,5p,twocolumn,times,10pt,lefttitle]{elsarticle}

\usepackage[numbers]{natbib}
\usepackage{graphicx,subfigure}      
\usepackage{enumerate}
\usepackage{tikz}
\usepackage{amsmath,amssymb,amsthm}

\begin{document}
\newtheorem{thm}{Theorem}
\newtheorem{cor}[thm]{Corollary}
\newtheorem{lem}[thm]{Lemma}
\newtheorem{claim}[thm]{Claim}
\newtheorem{axiom}[thm]{Axiom}
\newtheorem{conj}[thm]{Conjecture}
\newtheorem{fact}[thm]{Fact}
\newtheorem{hypo}[thm]{Hypothesis}
\newtheorem{assum}[thm]{Assumption}
\newtheorem{prop}[thm]{Proposition}
\newtheorem{crit}[thm]{Criterion}
\newtheorem{defn}[thm]{Definition}
\newtheorem{exmp}[thm]{Example}
\newtheorem{rem}[thm]{Remark}
\newtheorem{prob}[thm]{Problem}
\newtheorem{prin}[thm]{Principle}
\newtheorem{alg}{Algorithm}

\newcommand{\osd}{$(\mathbf{DSD})_{\mathbf{T}}$}
\newcommand{\ssd}{$\mathbf{SSD}$}
\newcommand{\sop}{$\mathbf{SOP}$}
\newcommand{\ocp}{$(\mathbf{OCP})_{\mathbf{T}}$}

\newcommand{\red}[1]{{\bf\color{red}{#1}}}
\newcommand{\blue}[1]{{\bf\color{blue}{#1}}}


\title{Shape turnpike for linear parabolic PDE models}                      

\author[1]{Gontran Lance}
\ead{lance@ljll.math.upmc.fr}

\author[1]{Emmanuel Tr\'elat}
\ead{emmanuel.trelat@sorbonne-universite.fr}

\author[2,3,4]{Enrique Zuazua}
\ead{enrique.zuazua@fau.de}

\address[1]{Sorbonne Universit\'e, CNRS, Universit\'e de Paris, Inria, Laboratoire Jacques-Louis Lions (LJLL), F-75005 Paris, France.}
\address[2]{Chair in Applied Analysis, Alexander von Humboldt-Professorship, Department of Mathematics  Friedrich-Alexander-Universit\"at, Erlangen-N\"urnberg, 91058 Erlangen, Germany.}
\address[3]{Chair of Computational Mathematics, Fundaci\'on Deusto Av. de las Universidades 24, 48007 Bilbao, Basque Country, Spain.}
\address[4]{Departamento de Matem\'aticas, Universidad Aut\'onoma de Madrid, 28049 Madrid, Spain.}



\begin{abstract}
We introduce and study the turnpike property for time-varying shapes, within the viewpoint of optimal control.
We focus here on second-order linear parabolic equations where the shape acts as a source term and we seek the optimal time-varying shape that minimizes a quadratic criterion.
We first establish existence of optimal solutions under some appropriate sufficient conditions. We then provide necessary conditions for optimality in terms of adjoint equations and, using the concept of strict dissipativity, we prove that state and adjoint satisfy the measure-turnpike property, meaning that the extremal time-varying solution remains essentially close to the optimal solution of an associated static problem. We show that the optimal shape enjoys the exponential turnpike property in term of Hausdorff distance for a Mayer quadratic cost.
We illustrate the turnpike phenomenon in optimal shape design with several numerical simulations.
\end{abstract}




\maketitle

\section{Introduction}
\label{sec:intro}
We start with an informal presentation of the turnpike phenomenon for general dynamical optimal shape problems, which has never been adressed in the litterature until now. Let $T>0$, we consider the problem of determining a time-varying shape $t \mapsto \omega(t)$ (viewed as a control, as in \cite{MR3350723}) minimizing the cost functional
\begin{equation}
J_T(\omega) =  \frac{1}{T}\int_0^T f^0\big(y(t),\omega(t)\big) \, dt + g\big(y(T),\omega(T)\big)
\label{shapemin}
\end{equation}
under the constraints
\begin{equation}
\dot{y}(t) = f\big(y(t),\omega(t)\big), \qquad R\big(y(0),y(T)\big) = 0
\label{pde}
\end{equation}
where (\ref{pde}) may be a partial differential equation with various terminal and boundary conditions.

We associate to the dynamical problem \eqref{shapemin}-\eqref{pde} a \emph{static} problem, not depending on time,
\begin{equation}
\displaystyle{\min_{\omega} f^0(y,\omega)}, \quad  f(y,\omega) = 0
\label{shape_static}
\end{equation}
i.e., the problem of minimizing the instantaneous cost under the constraint of being an equilibrium of the control dynamics.

According to the well known turnpike phenomenon, one expects that, for $T$ large enough, optimal solutions of \eqref{shapemin}-\eqref{pde} remain most of the time ``close" to an optimal (stationary) solution of the static problem (\ref{shape_static}). In this paper, we will investigate this problem in the linear parabolic case.

The turnpike phenomenon was first observed and investigated by economists for discrete-time optimal control problems (see \cite{turnpikefirst, 10.2307/1910955}). There are several possible notions of turnpike properties, some of them being stronger than the others (see \cite{MR3362209}). \emph{Exponential turnpike} properties have been established in \cite{GruneSchallerSchiela, MR3124890, MR3616131, TrelatZhangZuazua,MR3271298} for the optimal triple resulting of the application of Pontryagin's maximum principle, ensuring that the extremal solution (state, adjoint and control) remains exponentially close to an optimal solution of the corresponding static controlled problem, except at the beginning and at the end of the time interval, as soon as $T$ is large enough. This follows from hyperbolicity properties of the Hamiltonian flow. 
For discrete-time problems it has been shown in \cite{MR3217211, MR3654613, MR3782393, MR3470445, measureturnpikeTZ} that exponential turnpike is closely related to strict dissipativity.
\emph{Measure-turnpike} is a weaker notion of turnpike, meaning that any optimal solution, along the time frame, remains close to an optimal static solution except along a subset of times of small Lebesgue measure. It has been proved in \cite{MR3654613, measureturnpikeTZ} that measure-turnpike follows from strict dissipativity or from strong duality properties.

Applications of the turnpike property in practice are numerous. Indeed, the knowledge of a static optimal solution is a way to reduce significantly the complexity of the dynamical optimal control problem. For instance it has been shown in \cite{MR3271298} that the turnpike property gives a way to successfully initialize direct or indirect (shooting) methods in numerical optimal control, by initializing them with the optimal solution of the static problem. In shape design and despite of technological progress, it is easier to design pieces which do not evolve with time. Turnpike can legitimate such decisions for large-time evolving systems. 
 
\section{Shape turnpike for linear parabolic equation}
\label{sec:Shape Turnpike and heat equation}
Throughout the paper, we denote by:
\begin{itemize}
\item $Q \subset \mathbf{R}^d$, $d \geq 1$ and $|Q|$ its Lebesgue measure if $Q$ measurable subset;
\item $\big( p,q \big) $ the scalar product in $L^2(\Omega)$ of $p,q \mbox{ in } L^2(\Omega)$;
\item $\Vert y \Vert$ the $L^2$-norm of $y\in L^2(\Omega)$;
\item $\chi_{\omega}$ the indicator (or characteristic) function of $\omega \subset \mathbf{R}^d$;
\item $d_{\omega}$ the distance function to the set $\omega \subset \mathbf{R}^d$.
\end{itemize}

Let $\Omega \subset \mathbf{R}^d$ ($d \geq 1$) be an open bounded Lipschitz domain.
We consider the uniformly elliptic second-order differential operator 
$$
Ay=-\sum_{i,j=1}^d \partial_{x_j}\big(a_{ij}(x)\partial_{x_i}y\big)+\sum_{i=1}^d b_{i}(x)\partial_{x_i}y+c(x)y
$$
with $a_{ij},b_i \in C^1(\Omega)$, $c\in L^{\infty}(\Omega)$ with $c\geq 0$. We consider the operator $(A,D(A))$ defined on the domain $D(A)$ encoding Dirichlet conditions $y_{\vert\partial\Omega}=0$; when $\Omega$ is $C^2$ or a convex polytop in $\mathbf{R}^2$, we have $D(A)=H^1_0(\Omega)\cap H^2(\Omega)$. The adjoint operator $A^*$ of $A$, also defined on $D(A)$ with homogeneous Dirichlet conditions, is given by
$$
A^*v=-\sum_{i,j=1}^d\partial_{x_i}\left(a_{ij}(x)\partial_{x_j}v\right)-\sum_{i=1}^db_{i}(x)\partial_{x_i}v+\left(c-\sum_{i=1}^d\partial_{x_i}b_i\right)v
$$
and is also uniformly elliptic, see \cite[Definition Chapter 6]{MR2597943}. The operators $A$ and $A^*$ do not depend on $t$ and have a constant of ellipticity $\theta>0$ (for $A$ written in \textit{non-divergence form}), i.e.:
$$
\sum_{i,j=1}^d a_{ij}(x) \xi_i\xi_j \geq \theta \vert \xi \vert^2 \qquad \forall x \in \Omega.
$$
Moreover, we assume that
\begin{equation}\label{ineq_theta}
	\theta > \theta_1
\end{equation}
where $\theta_1$ is the largest root of the polynomial $P(X) = \frac{X^2}{4\min(1,C_p)} - \Vert c \Vert_{L^{\infty}(\Omega)} X - 
 \frac{\sum_{i=1}^d\Vert b_i\Vert_{L^{\infty}(\Omega)}}{2}$
with $C_p$ the Poincar\'e constant on $\Omega$. This assumption is used to ensure that an energy inequality is satisfied with constants not depending on the final time $T$ (see \ref{sec_app} for details).

We assume throughout that $A$ satisfies the classical maximum principle (see \cite[sec. 6.4]{MR2597943}) and that $c^*=c-\sum_{i=1}^d\partial_{x_i}b_i \in C^2(\Omega)$. 

Let $(\lambda_j, \phi_j)_{j \in \mathbf{N}^*}$ be the eigenelements of $A$ with $(\phi_j)_{j\in\mathbf{N}^*}$ an orthonormal eigenbasis of $L^2(\Omega)$:
\begin{itemize}
\item $ \forall j \in \mathbf{N}^{*},\qquad A \phi_{j} = \lambda_{j}\phi_{j}, \qquad \phi_{j_{\vert\partial\Omega}}=0$
\item $ \forall j \in \mathbf{N}^{*},\, j>1, \qquad \lambda_{1}< \lambda_{j} \leqslant \lambda_{j+1}, \qquad \lambda_{j}\rightarrow +\infty$.
\end{itemize}
A typical example satisfying all assumptions above is the Dirichlet Laplacian, which we will consider in our numerical simulations.

We recall that the Hausdorff distance between two compact subsets $K_1, K_2$ of $\mathbf{R}^d$ is defined by
$$
d_{\mathcal{H}}(K_1,K_2) = \sup\Big(\sup_{x\in K_2} d_{K_1}(x),\sup_{x\in K_1} d_{K_2}(x) \Big).
$$

\subsection{Setting}
Hereafter, we identify any measurable subset $\omega$ of $\Omega$ with its characteristic function $\chi_\omega$. Let $L \in (0,1)$. We define the set of admissible shapes
\begin{equation*}
\mathcal{U}_L = \{\omega \subset \Omega \mbox{ measurable } \mid \, \vert \omega \vert \leq L \vert \Omega \vert \}.
\end{equation*}

\paragraph{Dynamical optimal shape design problem \osd} 
Let $y_{0} \in L^{2}(\Omega)$ and let $\gamma_1 \geq 0, \gamma_2 \geq 0$ be arbitrary.
We consider the parabolic equation controlled by a (measurable) time-varying map $t\mapsto\omega(t)$ of subdomains
\begin{equation}
\partial_t y + A y = \chi_{\omega(\cdot)}, \qquad y_{\vert \partial \Omega}=0, \qquad y(0) = y_{0}.
\label{heat}
\end{equation}
Given $T>0$ and $y_d \in L^{2}(\Omega)$, we consider the dynamical optimal shape design problem \osd\, of determining a measurable path of shapes $t\mapsto \omega(t)\in \mathcal{U}_L$ that minimizes the  cost functional
\begin{equation*}
J_{T}(\omega(\cdot)) = \frac{\gamma_1}{2T}\int_{0}^{T}\Vert y(t)-y_{d}\Vert^{2}\,dt  + \frac{\gamma_2}{2}\,\Vert y(T) - y_d\Vert^2
\label{cost}
\end{equation*}
where $y=y(t,x)$ is the solution of (\ref{heat}) corresponding to $\omega(\cdot)$.

\paragraph{Static optimal shape design problem}
Besides, for the same target function $y_d \in L^2(\Omega)$, we consider the following associated static shape design problem:
\begin{equation}\tag{\ssd}
\displaystyle{\min_{\omega \in \mathcal{U}_L} \frac{\gamma_1}{2} \Vert y-y_{d}\Vert^{2}}, \quad A y =\chi_{\omega}, \quad y_{\vert \partial \Omega}=0.
\label{static}
\end{equation}
We are going to compare the solutions of \osd\, and of (\ssd)\, when $T$ is large.

\subsection{Preliminaries}
\paragraph{Convexification}
Given any measurable subset $\omega\subset\Omega$, we identify $\omega$ with its characteristic function $\chi_\omega\in L^\infty(\Omega;\{0,1\})$ and we identify $\mathcal{U}_L$ with a subset of $L^\infty(\Omega)$ (as in \cite{MR2745777, MR3325779, MR3500831}). Then, the convex closure of $\mathcal{U}_L$ in $L^\infty$ weak star topology is
$$
\overline{\mathcal{U}}_L = \Big\{ a \in L^{\infty}\big(\Omega;[0,1]\big)\ \mid\  \int_{\Omega}a(x)\,dx \leq L\vert\Omega\vert \Big\}
$$
which is also weak star compact. We define the \emph{convexified} (or \emph{relaxed}) optimal control problem \ocp\, of determining a control $t\mapsto a(t)\in \overline{\mathcal{U}}_L$ minimizing the cost 
$$
J_T(a)= \frac{\gamma_1}{2T}\int_{0}^{T}\Vert y(t)-y_{d}\Vert^{2}\,dt + \frac{\gamma_2}{2}\,\Vert y(T) - y_d\Vert^2 
$$
under the constraints
\begin{equation}
\partial_t y +A y = a, \qquad y_{\vert \partial \Omega}=0, \qquad y(0) = y_{0}. 
\label{heat_convex}
\end{equation}
The corresponding convexified static optimization problem is
\begin{equation}\tag{\sop}
\min_{a \in \overline{\mathcal{U}}_L}  \frac{\gamma_1}{2} \Vert y-y_{d}\Vert^{2}, \qquad
A y = a, \qquad y_{\vert \partial \Omega}=0.
\label{static_convex}
\end{equation}

Note that the control $a$ does not appear in the cost functionals of the above convexified control problems. Therefore the resulting optimal control problems are affine with respect to $a$. Once we have proved that an optimal solution $a \in \overline{\mathcal{U}}_L$ exists, we expect that any such minimizer will be an element of the set of extremal points of the compact convex set $\overline{\mathcal{U}}_L$, which is exactly the set $\mathcal{U}_L$ (since $\omega$ is identified with its characteristic function $\chi_{\omega}$). If this is true, then actually $a=\chi_\omega$ with $\omega\in\mathcal{U}_L$. Here, as it is usual in shape optimization, the interest of passing by the convexified problem is to allow us to derive optimality conditions, and thus to characterize the optimal solution. It is anyway not always the case that the minimizer $a$ of the convexified problem is an extremal point of $\overline{\mathcal{U}}_L$ (i.e., a characteristic function): in this case, we speak of a \emph{relaxation phenomenon}. Our analysis hereafter follows these guidelines. 

Taking a minimizing sequence and by classical arguments of functional analysis (see, e.g., \cite{MR0271512}), it is straightforward to prove existence of solutions $a_T$ and $\bar{a}$ respectively of \ocp\ and of (\sop)\, (see details in Section \ref{sec31}).

We underline the following fact: \textbf{if} $\bar a$ and $a_T(t),\mbox{ for } a.e.\, t\in[0,T]$, are characteristic functions of some subsets (meaning that $\bar a=\chi_{\bar\omega} \mbox{ with } \bar\omega\in\mathcal{U}_L$ and $\mbox{ for } a.e.\, t \in (0,T), a_T(t) = \chi_{\omega_T(t)} \mbox{ with } \omega_T(t)\in\mathcal{U}_L$), \textbf{then} actually $t\mapsto\omega_T(t)$ and $\bar\omega$ are optimal shapes, solutions respectively of \osd\, and of (\ssd).

Our next task is to apply necessary optimality conditions to optimal solutions of the convexified problems stated in \cite[Chapters 2 and 3]{MR0271512} or \cite[Chapter 4]{MR1312364} and infer from these necessary conditions that, under appropriate assumptions, the optimal controls are indeed characteristic functions.

\paragraph{Necessary optimality conditions for \ocp} 
According to the Pontryagin maximum principle (see \cite[Chapter 3, Theorem 2.1]{MR0271512}, see also \cite{MR1312364}), for any optimal solution $(y_T,a_T)$ of \ocp\ there exists an adjoint state $p_T \in L^{2}(0,T;\Omega)$ such that
\begin{align}
	&\begin{array}{r}
	\partial_t y_{T} + A y_{T} = a_{T},~ y_{T_{\vert \partial \Omega}}=0,~ y_{T}(0) = y_{0} \\[0.3cm]
	\hspace{-0.5cm}\partial_t p_{T} - A^*\!p_{T} \!=\! \gamma_1(y_T\!-\!y_d),~ p_{T_{\vert \partial \Omega}}\!=\!0, ~p_{T}(T) \!=\! \gamma_2 \big(y_T(T)\!-\!y_d\big)
	\label{OCocp} 
	\end{array} \\[0.2cm]
	\label{optim}
	&\forall a \in \overline{\mathcal{U}}_L, \textrm{for a.e.}\ t \in [0,T] : \quad \big(p_{T}(t),a_{T}(t)-a\big) \geq 0.
\end{align}

\paragraph{Necessary optimality conditions for (\sop)} 
Similarly, applying \cite[Chapter 2, Theorem 1.4]{MR0271512}, for any optimal solution $(\bar y,\bar a)$ of (\sop)\ there exists an adjoint state $\bar{p} \in L^{2}(\Omega)$ such that
\begin{eqnarray}
	\begin{array}{rcl}
	\hspace{1cm}A \bar{y} = \bar{a},&~& \bar{y}_{\vert \partial \Omega}=0 \\[0.2cm]
	\hspace{1cm}-A^* \bar{p} = \gamma_1(\bar{y}-y_d),&~& \bar{p}_{\vert \partial \Omega}=0
	\end{array}
	\label{OCsop}\\
	\forall a \in \overline{\mathcal{U}}_L : \quad (\bar{p},\bar{a} - a) \geq 0.
	\label{optimstat}
\end{eqnarray}
Using the bathtub principle (see, e.g., \cite[Theorem 1.14]{MR1817225}), (\ref{optim}) and (\ref{optimstat}) give
\begin{eqnarray}
\hspace{1cm}a_T(\cdot) &=& \chi_{\{p_T(\cdot) > s_T(\cdot)\}} + c_T(\cdot)\chi_{\{p_T(\cdot) = s_T(\cdot)\}}
\label{optimchi} \\
\hspace{1cm}\bar{a} &=& \chi_{\{\bar{p} > \bar{s}\}} + \bar{c}\chi_{\{\bar{p} = \bar{s}\}}
\label{optimstatchi}
\end{eqnarray} 
with, for a.e. $t\in[0,T]$,
\begin{eqnarray}
&c_T(t)& \!\!\!\!\in L^{\infty}(\Omega;[0,1]) \mbox{ and } \bar{c} \in L^{\infty}(\Omega;[0,1]) \\
&s_T(\cdot)& \!\!\!\!= \inf\big\{\sigma\in\mathbf{R}\ \mid\ \vert \{p_T(\cdot)>\sigma\} \vert  \leq L\vert \Omega \vert \big\} \label{leveltime} \\
&\bar{s}& \!\!\!\!= \inf\big\{\sigma\in\mathbf{R}\ \mid\ \vert\{\bar{p}>\sigma\}\vert  \leq L\vert \Omega \vert \big\}. \label{levelstat}
\end{eqnarray}

\subsection{Main results}
\paragraph{Existence of optimal shapes}
Proving existence of optimal shapes, solutions of \osd\, and of (\ssd), is not an easy task. Indeed, relaxation phenomena may occur, i.e., classical designs in $\mathcal{U}_L$ may not exist but may develop homogeneization patterns (see \cite[Sec. 4.2, Example 3]{henrot2005variation}). Therefore, some assumptions are required on the target function $y_d$ to establish existence of optimal shapes. We define:
\begin{itemize}
\item $y^{T,0} \mbox{ and } y^{T,1}$, the solutions of (\ref{heat_convex}) corresponding respectively to $a=0$ and $a=1$;
\item $y^{s,0} \mbox{ and } y^{s,1}$, the solutions of: $Ay=a,y_{\vert\partial\Omega}=0$, corresponding respectively to $a=0$ and $a=1$;
\item $\displaystyle{y^0 = \min \Big(y^{s,0}, \min_{t \in (0,T)} y^{T,0}(t)\Big)}$ and $ \displaystyle{y^1 = \max \Big(y^{s,1},\max_{t \in (0,T) }y^{T,1}\Big)}$.
\end{itemize}
We recall that $A$ is said to be analytic-hypoelliptic in the open set $\Omega$ if any solution of $Au=v$ with $v$ analytic in $\Omega$ is also analytic in $\Omega$. Analytic-hypoellipticity is satisfied for the second-order elliptic operator $A$ as soon as its coefficients are analytic in $\Omega$ (for instance it is the case for the Dirichlet Laplacian, without any further assumption, see \cite{Nelson}).
\begin{thm}
We distinguish between Lagrange and Mayer cases.
\begin{enumerate}
\item $\gamma_1=0, \gamma_2=1$ (Mayer case):
If $A$ is analytic-hypoelliptic in $\Omega$ then there exists a unique optimal shape $\omega_T$, solution of \osd.
\item $\gamma_1=1, \gamma_2=0$ (Lagrange case): Assuming that $y_0 \in D(A)$ 
and that $y_d \in H^2(\Omega)$:
\begin{enumerate}[(i)]
	\item If $y_d<y^0$ or $y_d>y^1$ then there exist unique optimal shapes $\bar\omega$ and $\omega_T$, respectively, of (\ssd)\, and of \osd.
	\item There exists a function $\beta$ such that if $A y_d \leq \beta$, then there exists a unique optimal shape $\bar \omega$, solution of (\ssd).
\end{enumerate}
\end{enumerate}
\label{existencethm}
\end{thm}

Proofs are given in Section \ref{sec:proof}. To prove existence of optimal shapes, we deal first with the convexified problems  \ocp\, and (\sop)\, and show existence and uniqueness of solutions. Hereafter, using optimality   conditions \eqref{OCocp}-\eqref{OCsop} and under the assumptions given in Theorem    \ref{existencethm} we can write the optimal control as characteristic functions of upper level sets of the adjoint variable. In the static case, for example, one key observation is to note that, if $\vert \big\{\bar{p} = \bar{s}\big\} \vert = 0$, then it follows from (\ref{optimstatchi}) that the static optimal control $\bar a$ is actually the characteristic function of a shape $\bar{\omega} \in \mathcal{U}_L$. This proves the existence of an optimal shape. 
\begin{rem}
Note that in the Mayer case ($\gamma_1 = 0,\gamma_2>0$), (\ssd)\, is reduced to solve $Ay=\chi_{\omega}$, $y_{\vert\partial\Omega}=0$. There is no criterion to minimize.
\end{rem}
\begin{rem}
Theorem \ref{existencethm} guarantees the uniqueness of an optimal shape. We deduce from the inequality \eqref{energy} in the appendix that we also have the uniqueness of the corresponding  state and adjoint state. Thus we have uniqueness for both the dynamic and the static optimal triple.
\end{rem}
In what follows, we study the behavior of optimal solutions of \osd\, compared to those of (\ssd)\, and give some turnpike properties. In the Lagrange case, inspired by \cite{MR3124890}, \cite{MR3616131} and \cite{measureturnpikeTZ}, we first prove that state and adjoint satisfy integral and measure turnpike properties. In the Mayer case, we estimate the Hausdorff distance between dynamical and static optimal shapes and show an exponential turnpike property. We denote by : 
\begin{itemize}
	\item $(y_T,p_T,\omega_T)$ the optimal triple of \osd\, and
	$$\displaystyle{J_T = \frac{\gamma_1}{2T}\int_{0}^{T}\Vert y_T(t)-y_{d}\Vert^{2}\,dt}+\frac{\gamma_2}{2}\Vert y_T(T)-y_{d}\Vert^{2};$$
	\item $(\bar{y},\bar{p},\bar{\omega})$ the optimal triple of (\ssd)\, and $\bar{J} = \frac{\gamma_1}{2}\Vert \bar{y}-y_{d}\Vert^{2}.$
\end{itemize}

\paragraph{Integral turnpike in the Lagrange case}
\begin{thm} For $\gamma_1=1, \gamma_2=0$ (Lagrange case), there exists $M>0$ (independent of the final time $T$) such that
$$
\int_0^T \big( \Vert y_T(t)-\bar{y} \Vert^2 + \Vert p_T(t)-\bar{p} \Vert^2 \big) \,dt \leq M\qquad \forall T>0.
$$
\label{integralturnpikethm}
\end{thm}

\paragraph{Measure-turnpike in the Lagrange case}
\begin{defn}\label{defmeasureturnpike}
We say that $(y_T,p_T)$ satisfies the \emph{state-adjoint measure-turnpike property} if for every $\varepsilon > 0$ there exists $\Lambda(\varepsilon)>0$, independent of $T$, such that
$$
\vert P_{\varepsilon,T} \vert < \Lambda(\varepsilon) \qquad \forall T >0
$$
where  $P_{\varepsilon,T} = \big\{t \in [0,T]\ \mid\ \Vert y_T(t)-\bar{y} \Vert+\Vert p_T(t)-\bar{p} \Vert>\varepsilon \big\} $.
\end{defn}

We refer to \cite{MR3155340,MR3654613,measureturnpikeTZ} (and references therein) for similar definitions. Here, $P_{\varepsilon,T}$  is the set of times along which the time optimal state-adjoint pair $\big(y_T,p_T\big)$ remains outside of an $\varepsilon$-neighborhood of the static optimal state-adjoint pair $(\bar{y},\bar{p})$ in $L^2$ topology.

Recall that a $\mathcal{K}$-class function is a continuous monotone increasing function $\alpha: [0;+\infty) \mapsto [0;+\infty)$ with $\alpha(0) = 0$. We now recall the notion of dissipativity (see \cite{MR0527462}).
\begin{defn}
We say that \osd\, is \emph{strictly dissipative} at an optimal stationary point $(\bar{y},\bar{\omega})$ of (\ref{static}) with respect to the \emph{supply rate function} 
$$
w(y,\omega) = \frac{1}{2}\Big(\Vert y-y_d \Vert^2 - \Vert \bar{y}-y_d \Vert^2 \Big)
$$
if there exists a \emph{storage function} $S:E\rightarrow \mathbf{R}$ locally bounded and bounded below and a \emph{$\mathcal{K}$-class function} $\alpha(\cdot)$ such that, for any $T >0$ and any $0<\tau<T$, the strict dissipation inequality
\begin{equation}
S(y(\tau)) + \int_0^\tau \alpha(\Vert y(t) - \bar{y} \Vert )\,dt \leq
S(y(0)) + \int_0^\tau  w\big(y(t),\omega(t)\big)\,dt 
\label{ineqDISSIP}
\end{equation}
is satisfied for any pair $\big(y(\cdot), \omega(\cdot)\big)$ solution of (\ref{heat}). 
\label{definitiondissipativity}
\end{defn} 

\begin{thm} 
For $\gamma_1=1, \gamma_2=0$ (Lagrange case):
\begin{enumerate}[(i)]
\item  \osd\, is strictly dissipative in the sense of Definition \ref{definitiondissipativity}.
\item The state-adjoint pair $(y_T,p_T)$ satisfies the measure-turnpike property.
\end{enumerate}
\label{measureturnpikethm}
\end{thm}

\paragraph{Exponential turnpike}
The exponential turnpike property is a stronger property and can be satisfied either by the state, by the adjoint or by the control or even by the three together. 
\begin{thm}
For $\gamma_1=0, \gamma_2=1$ (Mayer case): For $\Omega$ with $C^2$ boundary and $c=0$ there exist $T_0>0$, $M>0$ and $\mu>0$  such that, for every $T\geq T_0$,
\begin{equation*}
d_{\mathcal{H}}\big(\omega_T(t),\bar{\omega}\big) \leq M e^{-\mu(T-t)} \qquad  \forall t \in (0,T).
\end{equation*}
\label{turnpikeexpothm}
\end{thm}
In the Lagrange case, based on the numerical simulations presented in Section \ref{sec:Numerical simulations} we conjecture the exponential turnpike property, i.e., given optimal triples $(y_T,p_T,\omega_T)$ and $(\bar{y}, \bar{p}, \bar{\omega})$, there exist $C_1>0$ and $C_2>0$ independent of $T$ such that
\begin{equation*}
\Vert y_T(t)-\bar{y} \Vert +\Vert p_T(t)-\bar{p} \Vert+\Vert \chi_{\omega_T(t)}-\chi_{\bar{\omega}} \Vert \leq C_1 \Big(e^{-C_2 t} + e^{-C_2 (T-t)} \Big)
\end{equation*}
for a.e. $t \in [0,T]$.

\section{Proofs} \label{sec:proof}
\subsection{Proof of Theorem \ref{existencethm}} \label{sec31}
We first show existence of an optimal shape, solution for \ocp\, and similarly for (\sop). We first see that the infimum exists. We take a minimizing sequence $(y_{n},a_{n}) \in L^{2}(0,T;H^{1}_{0}(\Omega)) \times L^{\infty}\big(0,T;L^2\big(\Omega,[0,1]\big)\big)$ such that, for $n \in \mathbf{N}$, for $a.e.\, t \in[0,T], a_n(t) \in \overline{\mathcal{U}}_L$, the pair $(y_{n},a_{n})$ satisfies \eqref{heat_convex} and $J_T(a_n) \rightarrow J_T$. The sequence $(a_{n})$ is bounded in $L^{\infty}\big(0,T;L^2\big(\Omega,[0,1]\big)\big)$, so using \eqref{energy} and \eqref{gronwall}, the sequence $(y_{n})$ is bounded in $L^{\infty}(0,T;L^{2}(\Omega)) \cap L^{2}(0,T;H^{1}_{0}(\Omega))$. We show then, using \eqref{heat_convex}, that the sequence $(\frac{\partial y_{n}}{\partial t})$ is bounded in $L^{2}(0,T;H^{-1}(\Omega))$. We subtract a sequence still denoted by $(y_{n},a_{n})$ such that one can find a pair $(y,a) \in L^{2}(0,T;H^{1}_{0}(\Omega)) \times L^{\infty}\big(0,T;L^2\big(\Omega,[0,1]\big)\big)$ with 
\begin{eqnarray}
y_{n} &\rightharpoonup & y \,\,\,\qquad \text{weakly in }L^{2}(0,T;H^{1}_{0}(\Omega)) \nonumber\\
\partial_t y_{n} &\rightharpoonup &\partial_t y \,\,\,\quad \text{weakly in } L^{2}(0,T;H^{-1}(\Omega)) \nonumber \\
a_{n} &\rightharpoonup &a \qquad \text{weakly * in } L^{\infty}\big(0,T;L^2\big(\Omega,[0,1]\big)\big).
\label{weakstarcv}
\end{eqnarray}
We deduce that 
\begin{equation}
\begin{array}{rll}
\partial_t y_{n} +A y_{n} - a_{n} &\rightarrow & \partial_t y +A y - a \quad \text{in } \mathcal{D}'\big((0,T)\times\Omega\big) \\[0.2cm]
y_{n}(0) &\rightharpoonup & y(0) \quad ~\text{weakly in } L^{2}(\Omega). 
\end{array} 
\label{admissiblepair}
\end{equation}
We get using \eqref{admissiblepair} that $(y,a)$ is a weak solution of \eqref{heat_convex}.  Moreover, since $L^{\infty}\big(0,T;L^2\big(\Omega,[0,1]\big)\big) = \Big(L^{1}\big(0,T;L^2\big(\Omega,[0,1]\big)\big)\Big)'$ (see \cite[Corollary 1.3.22]{MR3617205}) the convergence \eqref{weakstarcv} implies that for every $v\in L^1(0,T)$ satisfying $v\geq 0$ and $\Vert v \Vert_{L^1(0,T)} = 1$, we have 
$\int_0^T \Big(\int_{\Omega} a(t,x)\,dx \Big)v(t)\,dt \leq L\vert\Omega\vert.$
Since the function $f_a$ defined by $f_a(t)=\int_{\Omega} a(t,x)\,dx$ belongs to $L^{\infty}(0,T)$, the norm $\Vert f_a \Vert_{L^{\infty}(0,T)}$ is the supremum of $\int_0^T \Big(\int_{\Omega} a(t,x)\,dx \Big)v(t)\,dt$ over the set of all possible $v \in L^{1}(0,T)$ such that $\Vert v\Vert_{L^{1}(0,T)} = 1$. Therefore $\Vert f_a \Vert_{L^{\infty}(0,T)} \leq L\vert\Omega\vert$ and $\int_{\Omega} a(t,x)\,dx \leq L\vert\Omega\vert$ for a.e. $t \in (0,T)$.
This shows that the pair $(y,a)$ is admissible. Since $H^{1}_{0}(\Omega)$ is compactly embedded in $L^{2}(\Omega)$ and by using the Aubin-Lions compactness Lemma (see \cite{aubinlions}), we obtain
$$
y_{n}  \rightarrow  y \quad \text{strongly in }L^{2}(0,T;L^{2}(\Omega)).
$$
We get then by weak lower semi-continuity of $J_T$ and by Fatou Lemma that
$$
J_T(a) \leq \lim \inf J_T(a_n).
$$
Hence $a$ is an optimal control for \ocp, that we rather denote by $a_{T}$ (and $\bar{a}$ for (\sop)). We next proceed by proving existence of optimal shape designs.

\textit{1-} We take $\gamma_1=0, \gamma_2=1$ (Mayer case). We consider an optimal triple $(y_T,p_T,a_T)$ of \ocp. Then it satisfies \eqref{OCocp} and \eqref{optimchi}. It follows from the properties of the parabolic equation and from the assumption of analytic-hypoellipticity that $p_T$ is analytic on $(0,T) \times \Omega$ and that all level sets $\{ p_T(t) = \alpha \}$ have zero Lebesgue measure. We conclude that the optimal control $a_T$ satisfying \eqref{OCocp}-\eqref{optimchi} is such that
\begin{equation}
\textrm{for a.e.} \ t \in [0,T]\quad \exists s(t) \in \mathbf{R}, \quad a_{T}(t,\cdot) = \chi_{\{p_{T}(t) > s(t)\}}
\label{aoptimal}
\end{equation}
i.e., $a_{T}(t)$ is a characteristic function.
Hence, for a Mayer problem \osd, existence of an optimal shape is proved. 

\textit{2-(i)} In the case $\gamma_1=1, \gamma_2=0$ (Lagrange case), we give the proof for the static problem (\ssd).
We suppose $y_d < y^0$ (we proceed similarly for $y_d > y^1$).
Having in mind (\ref{OCsop}) and (\ref{optimstatchi}), we have $A \bar{y} = \bar{c} \mbox{ on } \{ \bar{p} = \bar{s} \}$. 
By contradiction, if $\bar{c} \leq 1 \mbox{ on } \{ \bar{p} = \bar{s} \}$, let us consider the solution $y^*$ of: $Ay^*=a^*,y^*_{\vert\partial\Omega}=0$, with the control $a^*$ which is the same as $\bar{a}$ verifying (\ref{optimstatchi}) except that $\bar{c} = 0$ ($\bar{c} = 1$ if $y_d > y^1$) on $\{ \bar{p} = \bar{s} \}$. We have then $A(\bar{y}-y^*) \leq 0 $ (or $A(\bar{y}-y^*) \geq 0$ if $y_d > y^1$). Then, by the maximum principle (see \cite[sec. 6.4]{MR2597943}) and using the homogeneous Dirichlet condition, we get that the maximum (the minimum if $y_d > y^1$) of $\bar{y}-y^*$ is reached on the boundary and hence $y_d \geq y^* \geq \bar{y}$ (or $y_d \leq y^* \leq \bar{y}$ if $y_d > y^1$). We deduce $\Vert y^* - y_d \Vert \leq \Vert \bar{y} - y_d \Vert$. This means that $a^*$ is an optimal control. We conclude by uniqueness. 

We use a similar argument thanks to maximum principle for parabolic equations (see \cite[sec. 7.1.4]{MR2597943}) for existence of an optimal shape solution of \osd. 

In view of proving the next part of the theorem, we first give a useful Lemma inspired by \cite[Theorem 3.2]{MR3793605} and from \cite[Theorem 6.3]{MR3409135}.

\begin{lem}
Given any $p \in [1,+\infty)$ and any $u \in W^{1,p}(\Omega)$ such that $\vert \{u = 0\} \vert > 0$, we have $\nabla u = 0$ $a.e.$ on $\{u = 0\}$.
\label{derivativenullset}
\end{lem}

\begin{proof}[Proof of Lemma \ref{derivativenullset}.]
A proof of a more general result can be found in \cite[Theorem 3.2]{MR3793605}. For completeness, we give here a short argument.
$Du$ denotes here the weak derivative of $u$. We need first to show that for $u\in W^{1,p}(\Omega)$ and for a function $S \in C^1{(\mathbf{R})}$ for which there exists $M>0$ such that $\Vert S'\Vert_{L^{\infty}(\Omega)}<M$, we have $S(u)\in W^{1,p}(\Omega)$ and $DS(u)=S'(u)Du$. To do that, by the Meyers-Serrins theorem, we get a sequence $u_n \in C^{\infty}(\Omega)\cap W^{1,p}(\Omega)$ such that $u_n\rightarrow u$ in $ W^{1,p}(\Omega)$ and $u_n \rightarrow u$ almost everywhere. We get by the chain rule $DS(u_n) = S'(u_n)Du_n$ and $\int_{\Omega} \vert DS(u_n)\vert^p\,dx \leq \Vert S'\Vert^p_{{L^{\infty}(\Omega)}}\Vert Du_n\Vert^p_{L^p(\Omega)}$ involving $S(u_n) \in W^{1,p}(\Omega)$. Since $S$ is Lipschitz, we have $\Vert S(u_n) -S(u) \Vert_{L^p(\Omega)} \leq \Vert u_n-u\Vert_{L^p(\Omega)} \rightarrow 0 \mbox{ when } n\rightarrow 0$. We write then
\begin{align*}
&\Vert DS(u_n)-S'(u)Du\Vert_{L^p(\Omega)} = \Vert S'(u_n)Du_n-S'(u)Du\Vert_{L^p(\Omega)} \\[0.2cm]
&\leq \Vert S'(u_n)(Du_n-Du)\Vert_{L^p(\Omega)} + \Vert (S'(u_n)-S'(u))Du\Vert_{L^p(\Omega)}\\[0.2cm]
&\leq \Vert S'\Vert_{L^{\infty}(\Omega)}\Vert u_n - u \Vert_{W^{1,p}(\Omega)}+\Vert (S'(u_n)-S'(u))Du\Vert_{L^p(\Omega)}.
\end{align*}
The first term tends to $0$ since $u_n\rightarrow u$ in $ W^{1,p}(\Omega)$. For the second term, we use that $\vert S'(u_n)-S'(u)\vert^p\vert Du\vert^p \rightarrow 0$ a.e. and $\vert S'(u_n)-S'(u)\vert^p \vert Du \vert^p \leq 2\Vert S'\Vert^p_{L^{\infty}(\Omega)} \vert Du \vert^p \in L^1(\Omega)$. By the dominated convergence theorem, $\Vert (S'(u_n)-S'(u))Du\Vert_{L^p(\Omega)} \rightarrow 0$ which implies that $\Vert DS(u_n)-S'(u)Du\Vert_{L^p(\Omega)} \rightarrow 0$. Finally $S(u_n) \rightarrow S(u) \mbox{ in } W^{1,p}(\Omega)$ and $DS(u) = S'(u)Du$. Then, we consider $u^+ = \max(u,0)$ and $u^- = \min(u,0) = -\max(-u,0)$. We define 
$$
S_{\varepsilon}(s) = \left\{ \begin{array}{ll} (s^2+\varepsilon^2)^{\frac{1}{2}}-\varepsilon & \mbox{ if } s\geq 0 \\ 0 & \mbox{ else. } \end{array} \right.
$$
Note that $\Vert S_{\varepsilon}'\Vert_{L^{\infty}(\Omega)}<1$. We deduce that $DS_{\varepsilon}(u)=S_{\varepsilon}'(u)Du$ for every $\varepsilon>0$. For $\phi \in C^{\infty}_c(\Omega)$ we take the limit of $\int_{\Omega}S_{\varepsilon}(u)D\phi\,dx$ when $\varepsilon\rightarrow 0^+$ to get that 
$$
Du^+=\left\{\begin{array}{ll} Du &\mbox{ on } \{u>0\} \\0 &\mbox{ on }  \{u\leq 0\} \end{array} \right.
\ \textrm{and}\ 
Du^-=\left\{\begin{array}{ll} 0 &\mbox{ on } \{u\geq 0\} \\-Du &\mbox{ on }  \{u<0\} \end{array}. \right.
$$
Since $u = u^+-u^-$, we get $Du = 0$ on $\{u=0\}$. We can find this Lemma in a weaker form in \cite[Theorem 6.3]{MR3409135}.\end{proof}

\textit{2-(ii)} We assume that $A y_d \leq \beta$ in $\Omega$ with $\beta=\bar{s}Ac^*$. 
Having in mind (\ref{OCsop}) and (\ref{optimstatchi}), we assume by contradiction that $|\{\bar{p}=\bar{s}\}| > 0 $. Since $A$ and $A^*$ are differential operators, applying $A^*$ to $\bar{p}$ on $\{\bar{p}=\bar{s}\}$, we obtain by Lemma \ref{derivativenullset} that $A^* \bar{p} = c^*\bar{s}$ on $\{\bar{p}=\bar{s}\}$. Since $(\bar{y},\bar{p})$ verifies \eqref{OCsop} we get $y_d-\bar{y} = c^*\bar{s}$ on $\{\bar{p}=\bar{s}\}$. We apply then $A$ to this equation to get that $A y_d-\bar{s}Ac^* = A\bar{y} = \bar{a}$ on $\{\bar{p}=\bar{s}\}$. Therefore $A y_d-\bar{s}Ac^* \in (0,1)$ on $\{\bar{p}=\bar{s}\}$ which contradicts $A y_d \leq \beta$. Hence $|\{\bar{p}=\bar{s}\}| = 0 $ and thus \eqref{optimstatchi} implies $\bar{a}=\chi_{\bar{\omega}}$ for some ${\bar{\omega}} \in \mathcal{U}_L$. Existence of solution for (\ssd)\, is proved. 

The uniqueness of optimal controls comes from the strict convexity of the cost functionals. Indeed, in the dynamical case, whatever $(\gamma_1, \gamma_2)\neq(0,0)$ may be, $J_T$ is strictly convex with respect to variable $y$. The injectivity of the control-to-state mapping gives the strict convexity with respect to the variable $a$. In addition, uniqueness of $(\bar{y},\bar{p})$ follows by application of the Poincar\'e inequality and uniqueness of $(y_T,p_T)$ follows from Gronwall inequality \eqref{gronwall} in the appendix.

\subsection{Proof of Theorem \ref{integralturnpikethm}}
For $\gamma_1=1, \gamma_2=0$ (Lagrange case), the cost is $J_{T}(\omega) = \frac{1}{2T}\int_{0}^{T}\Vert y(t)-y_{d}\Vert^{2}\,dt$. We consider the triples $(y_T,p_T,\chi_{\omega_T})$ and $(\bar{y},\bar{p},\chi_{\bar{\omega}})$ satisfying the optimality conditions (\ref{OCocp}) and (\ref{OCsop}).
Since $\chi_{\omega_T(t)}$ is bounded at each time $t \in [0,T]$ and by application of Gronwall inequality \eqref{gronwall} in the appendix to $y_{T}$ and $p_{T}$ we can find a constant $C>0$ depending only on $A, y_0, y_d, \Omega, L$ such that
\begin{equation*}
\forall T >0 \quad \Vert y_T(T) \Vert^{2} \leq C \quad \mbox{and} \quad
\Vert p_T(0) \Vert^{2} \leq C.
\end{equation*}
Setting $\tilde{y} = y_T-\bar{y},\tilde{p} =p_T-\bar{p},\tilde{a}=\chi_{\omega_T}-\chi_{\bar{\omega}}$, we have
\begin{eqnarray}
\partial_t \tilde{y} +A\tilde{y} = \tilde{a}, \quad \tilde{y}_{\vert \partial \Omega}&=&0, \quad \tilde{y}(0) = y_{0}-\bar{y}
\label{optimedpy} \\
\partial_t \tilde{p} -A^* \tilde{p} = \tilde{y}, \quad \tilde{p}_{\vert \partial \Omega}&=&0,  \quad \tilde{p}(T) = -\bar{p}.
\label{optimedpp}
\end{eqnarray}
First, using (\ref{OCocp}) and (\ref{OCsop}) one has $\big(\tilde{p}(t),\tilde{a}(t)\big) \geq 0$ for almost every $t \in [0,T]$.
Multiplying (\ref{optimedpy}) by $\tilde{p}$, (\ref{optimedpp}) by $\tilde{y}$ and then adding them, one  can use the fact that
$$
\big(\bar{y}-y_{0},\tilde{p}(0)\big) - \big(\tilde{y}(T),\bar{p}\big) = \int_{0}^{T}\big(\tilde{p}(t),\tilde{a}(t)\big)\,dt + \int_{0}^{T}\Vert \tilde{y}(t) \Vert^{2} \,dt.
\label{ineq_yp}
$$
By the Cauchy-Schwarz inequality we get a new constant $C>0$ such that
\begin{equation*} 
\frac{1}{T}\int_{0}^{T}\Vert \tilde{y}(t) \Vert^{2} \,dt +\frac{1}{T} \int_{0}^{T} \big(\tilde{p}(t),\tilde{a}(t)\big)\,dt \leq \frac{C}{T}.
\end{equation*}
The two terms at the left-hand side are positive and using the inequality (\ref{energy}) with $\zeta(t) = \tilde{p}(T-t)$, we finally obtain $M>0$ independent of $T$ such that
\begin{equation*}
\frac{1}{T}\int_{0}^{T} \big(\Vert y_T(t) - \bar{y} \Vert^{2} + \Vert  p_T(t) - \bar{p} \Vert^{2}\big) \,dt \leq \frac{M}{T}. 
\end{equation*}

\subsection{Proof of Theorem \ref{measureturnpikethm}}
(i) Strict dissipativity is established thanks to the storage function $S(y) = \big(y,\bar{p}\big)$ where $\bar{p}$ is the optimal adjoint. Following the Gronwall inequality \eqref{gronwall} in the appendix, since $\Vert y(t) \Vert^2 < M$ for every $t \in[0,T]$ with $M$ independent of final time $T$, the storage function $S$ is locally bounded and bounded below. We next consider an admissible pair $(y(\cdot),\chi_{\omega(\cdot)})$ of \ocp, we multiply \eqref{heat} by $\bar{p}$ and or $\tau>0$, we integrate over $(0,\tau)\times\Omega$ and use optimality conditions of static problem \eqref{OCsop}-\eqref{optimstat} combined with integration by parts to write
\begin{align*}
\int_0^{\tau} \big(y_t +Ay,\bar{p}\big)\,dt = \int_0^{\tau} \big(\chi_{\omega(t)},\bar{p}\big)\,dt \leq  \int_0^{\tau} \big(\chi_{\bar{\omega}},\bar{p}\big)\,dt \\
\mbox{and so}\qquad \big(y(\tau),\bar{p}\big) - \int_0^{\tau} \big(y(t)-\bar{y},\bar{y}-y_d\big)\,dt \leq \big(y(0),\bar{p}\big).
\end{align*}%
Noting that $\Vert y-\bar{y}\Vert^2 +2\big(y-\bar{y},\bar{y}-y_d\big) = \Vert y-y_d \Vert^2 - \Vert \bar{y}-y_d\Vert^2$ we make appear the quantity $\Vert y(t)-\bar{y}\Vert^2$ and finally get the strict dissipation inequality (\ref{ineqDISSIP}) with respect to the supply rate function $w(y,\omega) = \frac{1}{2}\Big(\Vert y-y_d \Vert^2 - \Vert \bar{y}-y_d\Vert^2\Big)$ and with $\alpha(s)=\frac{1}{2}s^2$:
\begin{equation}
(\bar{p},y(\tau)) + \int_0^\tau \frac{1}{2}\Vert y(t) - \bar{y} \Vert^2 \,dt \leq (\bar{p},y(0)) + \int_0^\tau  w\big(y(t),\omega(t)\big)\,dt.
\label{dissipativityinequality}
\end{equation}

(ii) Now we prove that strict dissipativity implies measure-turnpike, by following an argument of \cite{measureturnpikeTZ}.
Applying (\ref{dissipativityinequality}) to the optimal solution $(y_T,\omega_T)$ at $\tau=T$, we get
\begin{equation}
\frac{1}{T} \int_0^T \Vert y_T(t)-\bar{y} \Vert^2 \, dt   \leq J_T-\bar{J} + \frac{(y_T(0)-y_T(T),\bar{p})}{T}.
\nonumber
\end{equation}
Considering then the solution $y_s$ of (\ref{heat}) with $\omega(\cdot) = \bar{\omega}$ and $J_s = {\frac{1}{T}\int_{0}^{T}\Vert y_s(t)-y_{d}\Vert^{2}}$, we have $J_T-J_s < 0$ and we show that $J_s-\bar{J} \leq \frac{1-e^{-CT}}{CT}$, then we find $M_1>0$ independent of $T$ such that
\begin{equation} 
\frac{1}{T} \int_0^T \Vert y_T(t)-\bar{y} \Vert^2 \, dt   \leq  \frac{M_1}{T}.
\label{dissipativityineqstate}
\end{equation}
Applying (\ref{energy}) to $\zeta(\cdot) = p_T(T-\cdot) - \bar{p}$, we get $M_2>0$ independent of $T$ such that
\begin{equation}
\frac{1}{T} \int_0^T \Vert p_T(t)-\bar{p} \Vert^2 \,dt\leq \frac{M_2}{T} \int_0^T \Vert y_T(t)-\bar{y} \Vert^2 \,dt.
\label{dissipativityineqadjoint}
\end{equation}
We combine \eqref{dissipativityineqstate} and \eqref{dissipativityineqadjoint} to finally get a constant $M > 0$ which does not depend on $T$ such that $\forall \varepsilon >0, ~ \displaystyle{\vert P_{\varepsilon,T} \vert \leq \frac{M}{\varepsilon^2}}$.

\subsection{Proof of Theorem \ref{turnpikeexpothm}}
We take $\gamma_1=0,\gamma_2=1$ (Mayer case). We want to characterize optimal shapes as being the level set of some functions as in \cite{dambrine:hal-02057510}.
Let $(y_T,p_T,\chi_{\omega_T})$ be an optimal triple, coming from Theorem \ref{existencethm}-(i). Then $\zeta(t,x) = p_{T}(T-t,x)$ satisfies
\begin{equation}
\partial_t \zeta +A^* \zeta = 0, \quad \zeta_{\vert \partial \Omega}=0, \quad \zeta(0) = y_d-y_{T}(T).
\label{adjointretrograde}
\end{equation}
We write $y_d-y_T(T)$ in the basis $(\phi_j)_{j \in \mathbf{N}^*}$. There exists $(\zeta_j) \in \mathbf{R}^{\mathbf{N}^*}$ such that $ y_d - y_T(T) = \sum_{j\geq 1} \zeta_j \phi_j $. We can solve \eqref{adjointretrograde} and get $ p_T(t,x) = \sum_{j\geq 1} \zeta_j \phi_j(x) e^{-\lambda_j(T-t)}$. Using the Gronwall inequality \eqref{gronwall} in the appendix, there exists $C_1>0$ independent of $T$ such that the solution of \eqref{heat} satisfies $\Vert y_T(t) \Vert ^2 \leq C_1$ for every $t \in (0,T)$. Hence $\vert \zeta_j \vert^2 \leq C_1$ for every $j \in \mathbf{N}^*$. Let us consider the index $j_0 = \inf\{j \in \mathbf{N}, \zeta_j \neq 0\}$. Take $\lambda = \lambda_{j_0}$ and $\mu = \lambda_k$ where $k$ is the first index for which $\lambda_k>\lambda$. We define $\displaystyle{\Phi_0 = \sum_{\lambda_j = \lambda_{j_0}} \zeta_j\phi_j}$ which is a finite sum of the eigenfunctions associated to the eigenvalue $\lambda_{j_0}$. We write, for every $x \in \Omega$ and every $t \in [0,T]$,
\begin{align*}
\vert p_T(t,x) - e^{-\lambda(T-t)}\Phi_0(x) \vert &= \left\vert \sum_{j\geq k} \zeta_j \phi_j(x) e^{-\lambda_j (T-t)} \right\vert \\
&\leq \sum_{j\geq k} \left\vert \zeta_j \phi_j(x)\right\vert e^{-\lambda_j (T-t)}. 
\end{align*}
Since $\vert \zeta_j \vert^2 \leq C_1, \forall j \in \mathbf{N}^*$, by the Weyl Law and sup-norm estimates for the  eigenfunctions of $A$ (see \cite[Chapter 3]{MR3186367}), we can find $\alpha \in (0,1)$ such that $\alpha \mu > \lambda$ and two constants $C_1,C_2>0$ independent of $x$, $t$ and $T$ such that
$$
\vert p_T(t,x) - e^{-\lambda(T-t)}\Phi_0(x) \vert \leq C_1 e^{-\alpha\mu(T-t)} \sum_{j\geq k}   j^{\frac{N-1}{2N}} e^{-C_2 j^{\frac{1}{N}}(T-t)}.
$$
Let $\varepsilon > 0$ be arbitrary. We claim that there exists $C_{\varepsilon}>0$ independent of $x$, $t$, $T$ such that, for every $ x \in \Omega$,
\begin{align*}
\vert p_T(t,x) - e^{-\lambda(T-t)}\Phi_0(x) \vert &\leq C_{\varepsilon} e^{-\alpha\mu(T-t)}\quad \forall t \in (0,T-\varepsilon) \\[0,2cm]
\vert p_T(t,x) - e^{-\lambda(T-t)}\Phi_0(x) \vert &\leq C_{\varepsilon} \qquad \forall t \in (T-\varepsilon,T).
\end{align*}
To conclude we take an arbitrary value for $\varepsilon$ and we write $\mu$ instead of $\alpha \mu$ but always with $\mu>\lambda$ to get
\begin{equation}
\Vert p_T(t)-e^{-\lambda(T-t)}\Phi_0 \Vert_{L^{\infty}(\Omega)} \leq C\,e^{-\mu (T-t)} \quad \forall t \in [0,T]
\label{turnpikeadjoint}
\end{equation}
with $C>0$ not depending on the final time $T$. 
Using the bathtub principle (\cite[Theorem 1.16]{MR1817225}) and since $\Phi_0$ is analytic, we introduce $s_0 \in \mathbf{R}$ and the shape $\omega_{0} = \{\Phi_0 \geq s_{0}\} \in \mathcal{U}_L$ such that $\chi_{\omega_{0}}$ is solution of the auxiliary problem
\begin{equation}
\max_{u\in \overline{\mathcal{U}}_L} \int_{\Omega} \Phi_0(x)u(x)\,dx.  
\label{solstatic}
\end{equation}
Let $t\in[0,T]$ fixed. For $x\in \omega_0$, we remark that \eqref{turnpikeadjoint} implies that $p(t,x) \geq  s_0 e^{-\lambda(T-t)} - C^{-\mu(T-t)}$. Reminding the definition of $s_T(t)$ in \eqref{leveltime} we write
$$
\left\{\begin{array}{l}
\omega_0 \subset \big\{p(t,x) \geq s_0 e^{-\lambda(T-t)} - C^{-\mu(T-t)}\big\} \\[0.2cm]
\vert \omega_0 \vert = L\vert \Omega \vert \quad \mbox{and} \quad \Big\vert \big\{ p_T(t,x) \geq s_T(t) \big\} \Big\vert \leq L\vert \Omega \vert.
\end{array} \right. 
$$
Hence $s_T(t) \geq  s_0 e^{-\lambda(T-t)} - C^{-\mu(T-t)}$.
We change the roles of $\omega_0$ and $\omega_T(t)$ to get $s_T(t) \leq  s_0 e^{-\lambda(T-t)} + C^{-\mu(T-t)}$ and finally obtain
\begin{equation}
\vert s_T(t) - e^{-\lambda(T-t)} s_0  \vert \leq C\,e^{-\mu (T-t)} \quad \forall t \in [0,T].
\label{turnpikelevel}
\end{equation}
We write $\Phi = s_0-\Phi_0$, $\psi_T(t,x) =s_T(t) - p_T(t,x)$ and $\psi_0(t,x) = e^{-\lambda(T-t)}\Phi(x)$ and using \eqref{turnpikeadjoint} with \eqref{turnpikelevel}, we get a new constant $C>0$ independent of $T$ such that
\begin{equation}
\Vert \psi_T(t,x)-\psi_0(t,x)\Vert_{L^{\infty}(\Omega)} \leqslant C\,e^{-\mu (T-t)}, \quad \forall t \in [0,T].
\label{turnpikepsi}
\end{equation}
We now follow arguments of \cite{dambrine:hal-02057510} to establish the exponential turnpike property for the control and then for the state by using some information on the control $\chi_{\omega_T}$. We first remark that for all $t_1,t_2 \in [0,T]$, $\big\{\psi_0(t_1,\cdot) \leq 0\big\} = \big\{\psi_0(t_2,\cdot) \leq 0\big\}  = \big\{ \Phi \leq 0 \big\} $. Then we take $t \in [0,T]$ and we compare the sets $\big\{\psi_0(t,\cdot) \leq 0\big\}, \big\{\psi_T(t,\cdot) \leq 0\big\} \mbox{ and } \big\{\psi_0(t,\cdot) + C e^{-\mu (T-t)} \leq 0\big\}$. Thanks to \eqref{turnpikepsi} we get for every $ t \in [0,T]$
\begin{eqnarray}
\hspace{-0.5cm} \big\{\Phi\! \leq\! -C e^{-(\mu-\lambda) (T-t)} \big\} \subset \big\{\psi_T(t,\cdot)\! \leq \!0\big\} \subset \big\{\Phi\! \leq\! C e^{-(\mu-\lambda) (T-t)}\big\}~ \\
\hspace{-0.5cm} \big\{\Phi\! \leq \!-C e^{-(\mu-\lambda) (T-t)} \big\} \subset \big\{\psi_0(t,\cdot) \!\leq\! 0\big\} \subset \big\{\Phi \!\leq\! C e^{-(\mu-\lambda) (T-t)}\big\}.
\end{eqnarray}
We infer from \cite[Lemma 2.3]{dambrine:hal-02057510} that for every $ t \in [0,T]$,
\begin{multline}\label{inegalitehausdorffdistance}
d_{\mathcal{H}} \Big( \big\{\psi_T(t,\!\cdot) \leq  0\big\}, \big\{\Phi \leq 0\big\} \Big) \\
\leq d_{\mathcal{H}} \Big( \big\{\Phi\leq\!-C e^{-(\mu-\lambda) (T-t)}\big\},
 \big\{\Phi \leq C e^{-(\mu-\lambda) (T-t)}\big\}\Big).
\end{multline}
To conclude, since $d_{\mathcal{H}}$ is a distance, we only have to estimate $d_{\mathcal{H}} \Big( \big\{\Phi\!\leq\!0\big\},\big\{\Phi\! \leq \pm C e^{-(\mu-\lambda) (T-t)}\big\}\Big)$. 

\begin{lem}
Let $f : \Omega \rightarrow \mathbf{R}$ be a continuously differentiable function and set $\Gamma = \big\{f=0\big\}$. Under the assumption \emph{\emph{(S)}}: there exists $C>0$ such that  
$$
\Vert \nabla f(x) \Vert \geq C \quad \forall x \in \Gamma,
$$
there exist $\varepsilon_0>0$ and $C_f>0$ only depending on $f$ such that for any $\varepsilon \leq \varepsilon_0$ 
$$
d_{\mathcal{H}}\big(\big\{f\leq 0\big\},\big\{f \leq \pm \varepsilon\big\}\big) \leq C_f \varepsilon .
$$
\label{propdambrine}
\end{lem}

\begin{proof}[Proof of Lemma \ref{propdambrine}.]
We consider $f$ satisfying \emph{(S)} with $\Gamma = \big\{ \Phi =0 \big\}$. We assume by contradiction that for every $\varepsilon >0$, there exists $x \in \big\{\vert f\vert \leq \varepsilon \big\}$ such that $\Vert \nabla f(x) \Vert =0$. We take $\varepsilon = \frac{1}{n}$ and we subtract a subsequence $(x_n)\rightarrow x \in \big\{\vert f\vert \leq 1 \big\}$ (which is compact). By continuity of $f$ and of $\Vert \nabla f\Vert$, we have $x \in \Gamma$ and $\Vert f(x) \Vert=0$, which raises contradiction with \emph{(S)}. Hence we find $\varepsilon_0>0$ such that $\Vert \nabla f(x) \Vert \geq \frac{C}{2}$ for every $x \in \big\{\vert f\vert \leq \varepsilon \big\}$. We apply \cite[Corollary 4]{MR2592958} (see also \cite[Theorem 2]{MR2592958}) to get
$$
d_{\mathcal{H}}\big(\big\{f\leq 0\big\},\big\{f \leq \pm \varepsilon\big\}\big) \leq \frac{2}{C} \varepsilon.
$$
A more general statement can be found in \cite{MR2592958, dambrine:hal-02057510}.
\end{proof}

We infer that $\Phi$ satisfies \emph{(S)} on $\Vert \nabla_x \psi_0(t,x) \Vert = e^{-\lambda(T-t)}\Vert \nabla_x \Phi(x) \Vert $ for $x \in \Omega$. We first remark that $\Phi_0$ satisfies $A \Phi_0 = \lambda_{j_0} \Phi_0, \Phi_{0_{\vert \Gamma}} = s_0$ and that the set $\Gamma=\big\{\Phi_0=0\big\}$ is compact. Since $\Omega$ has a $C^2$ boundary and $c=0$ the Hopf lemma (see \cite[sec. 6.4]{MR2597943}) gives
$$
x_0 \in \Gamma_0 \implies  \Vert \nabla_x \Phi (x_0) \Vert = \Vert \nabla_x \Phi_0 (x_0) \Vert > 0.
$$ 
Hence there exists $C_0>0$ not depending on $t$, $T$ such that for every $x \in \Gamma_0$, $\Vert \nabla_x \Phi (x_0) \Vert \geq C_0 >0$. We take $\nu>0, e^{-\mu \nu} \leq \varepsilon_0$. We remark that $e^{-\mu (T-t)} \leq \varepsilon_0, \forall t \in (0,T-\nu)$ and we use Lemma \ref{propdambrine} combined with \eqref{inegalitehausdorffdistance} to get that, for every $t \in (0,T-\nu)$,
$$
d_{\mathcal{H}} \Big( \big\{\psi_T(t,\!\cdot) \leq  0\big\}, \big\{\Phi \leq 0\big\} \Big) \leq 
C_0 e^{-(\mu-\lambda) (T-t)}.
$$
We adapt the constant $C_0$ such that on the compact interval $t\in (T-\nu,T)$ the sets are the same whatever $T\geq T_0>0$ may be, to get that, for every $t \in (0,T)$,
$$
d_{\mathcal{H}} \Big( \big\{\psi_T(t,\!\cdot) \leq  0\big\}, \big\{\Phi \leq 0\big\} \Big) \leq 
C_0 e^{-(\mu-\lambda) (T-t)}.
$$
We obtain therefore an exponential turnpike property for the control in the sense of the Hausdorff distance 
\begin{equation}
d_{\mathcal{H}} ( \omega_T(t), \omega_0 )\leq C_0 e^{-(\mu-\lambda) (T-t)} \quad \forall t \in [0,T].
\label{turnpikeshape}
\end{equation}
Here is a possible way to find a further turnpike property on state and adjoint. We could use a similar argument (valid only for convex sets) as in \cite[Theorem 1-(i)]{MR1745583}: $\Vert \chi_{\omega_T(t)} - \chi_{\omega_0} \Vert \leq C d_{\mathcal{H}} ( \omega_T(t), \omega_0 )$. Denoting by $b_{\omega} = d_{\omega}-d_{\omega^c}$ the oriented distance, we follow \cite[Theorem 4.1-(ii)]{MR1855817} and \cite[Theorem 5.1-(iii)(iv)]{MR1855817} and we use the inequality $\Vert \chi_{\overline{A}_1} - \chi_{\overline{A}_2}\Vert \leq \Vert d_{A_1} - d_{A_2} \Vert_{W^{1,2}(\Omega)} \leq \Vert b_{A_1} - b_{A_2} \Vert_{W^{1,2}(\Omega)} = \Vert b_{A_1} - b_{A_2} \Vert + \Vert \nabla b_{A_1} - \nabla b_{A_2} \Vert$ to try to make the link between $\Vert \chi_{\omega_T(t)} - \chi_{\omega_0} \Vert$ and $d_{\mathcal{H}} ( \omega_T(t), \omega_0 )$. Afterwards, applying Gronwall inequality \eqref{gronwall}, we get
\begin{equation}
\Vert y(t) - \bar{y} \Vert_{L^{2}(\Omega)} \leq C_0 e^{-\frac{(\mu+\lambda)}{2} (T-t)} \quad \forall t \in (0,T)
\label{turnpikestate}
\end{equation}
with $\bar{y}$ solution of $A y = \chi_{\omega_0}, y_{\vert \partial \Omega} = 0$. Taking $\kappa =  \frac{\mu+\lambda}{2} > 0$ and by application of Gronwall inequality \eqref{gronwall} for the adjoint, we finally get the exponential turnpike property for the state, adjoint and control.

\section{Numerical simulations: optimal shape design for the 2D heat equation}
\label{sec:Numerical simulations}
We take $\Omega = [-1,1]^2$, $L = \frac{1}{8}$, $T\in\{1,\ldots,5\}$, $y_{d}=\mathrm{Cst}=0.1$ and $y_0=0$.
We focus on the heat equation and consider the minimization problem
\begin{equation}
\displaystyle{\min_{\omega(\cdot)} \int_{0}^{T}{\int_{[-1,1]^2} |y(t,x)-0.1|^{2}\, dx\, dt}}
\end{equation}
under the constraints
\begin{equation}
\partial_t y - \triangle y = \chi_{\omega}, \qquad y(0,\cdot) = 0,\qquad y_{\vert \partial \Omega}=0.
\end{equation}

We compute numerically a solution by solving the equivalent convexified problem \ocp\, thanks to a direct method in optimal control (see \cite{MR2224013}). We discretize here with an implicit Euler method in time and with a decomposition on a finite element mesh of $\Omega$ using \texttt{FREEFEM++} (see \cite{MR3043640}). We express the problem as a quadratic programming problem in finite dimension. We use then the routine \texttt{IpOpt} (see \cite{ipopt}) on a standard desktop machine.

\begin{figure}[!h]
\centering
\includegraphics[width=1\linewidth]{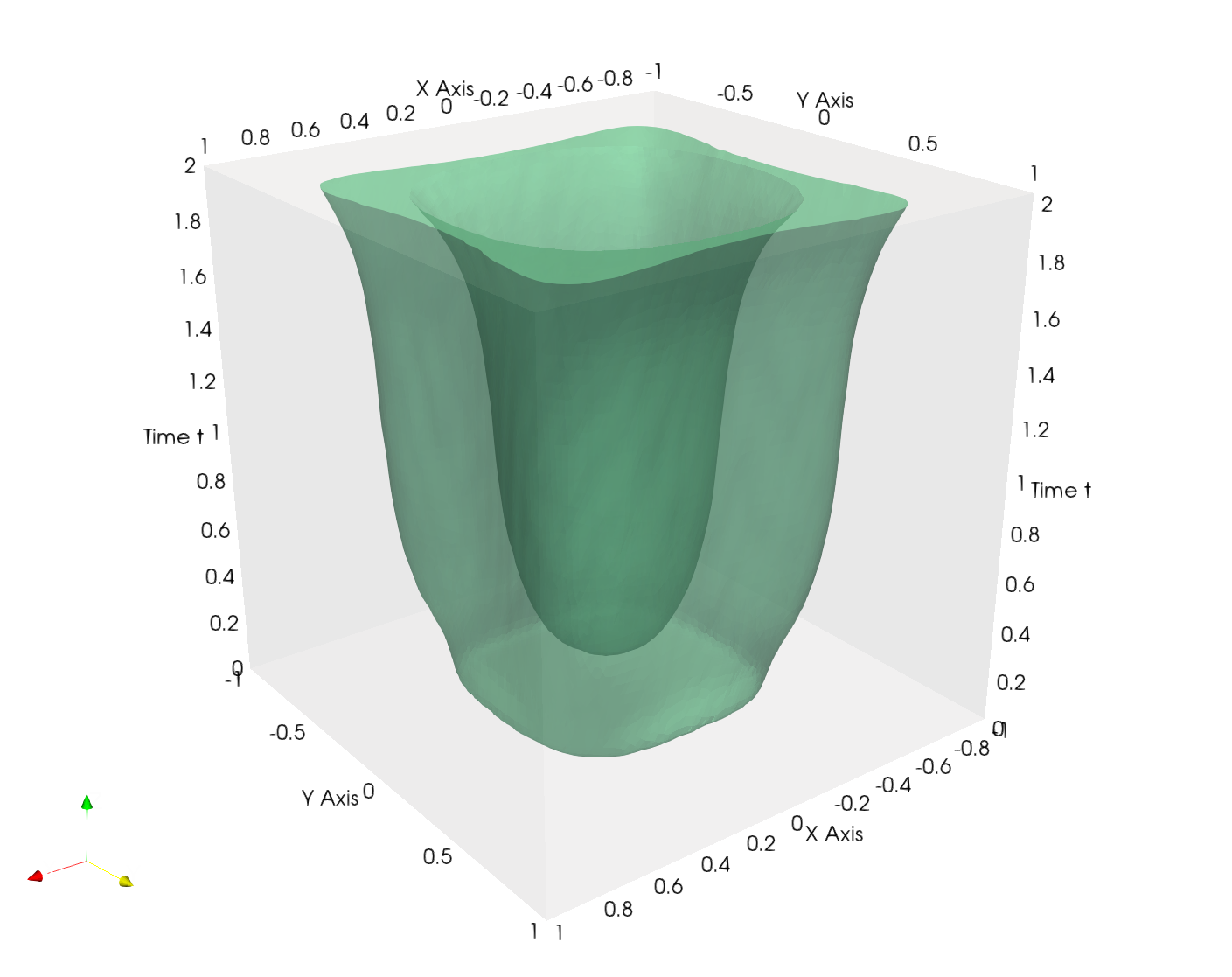}
\caption{Optimal shape's time evolution cylinder - $T=2$}
\label{fig:colonne}
\end{figure}

We plot in Figure \ref{fig:colonne} the evolution in time of the optimal shape $t\mapsto\omega(t)$ which appears like a cylinder whose section at time $t$ represents the shape $\omega(t)$. At the beginning $(t=0$) we notice that the shape concentrates at the middle of $\Omega$ in order to warm as soon as possible near to $y_d$. Once it is acceptable the shape is almost stationary during a long time. Finally, since the target $y_d$ is taken here as a constant, the optimal final state $y_T(T)$ should be as flat as possible. Indeed, for  $t<T$ and plotting the state's curve, we observe that $y_T(t)$ is much larger at the center of $\Omega$ than close to the boundary. So at final time, the shape comes closer to the boundary of $\Omega$ such that $y_T(T)$ gets larger close to it and lower at the center. We observe therefore that $y_T(T)$ is almost constant in $\Omega$ and very close to $y_d$.

~~

\begin{figure}[!h]%
\centering
\subfigure[][]{%
\label{fig:ex3-a}%
\includegraphics[scale=0.145]{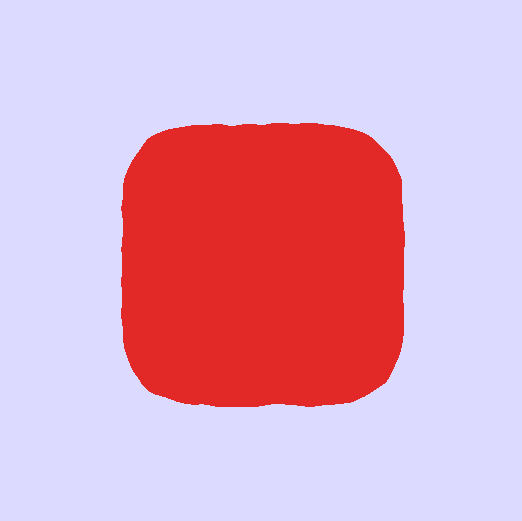}}%
\hspace{25pt}%
\subfigure[][]{%
\label{fig:ex3-b}%
\includegraphics[scale=0.145]{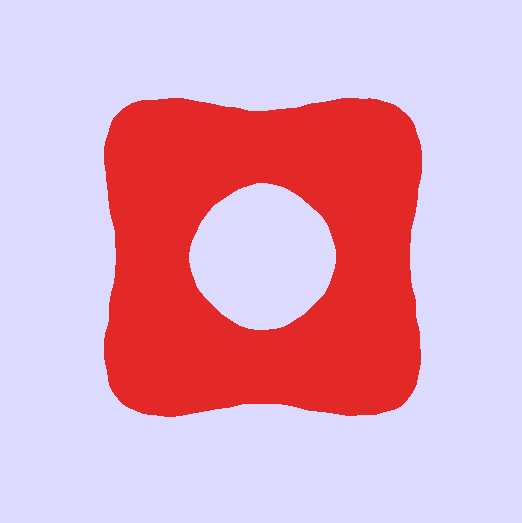}}%
\hspace{25pt}%
\subfigure[][]{%
\label{fig:ex3-c}%
\includegraphics[scale=0.145]{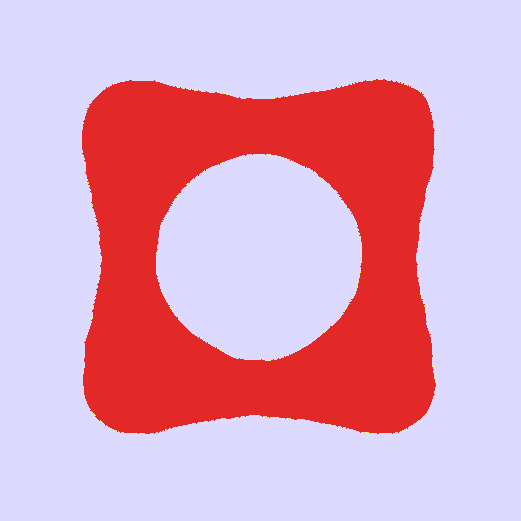}} \\\vspace{5pt}
\subfigure[][]{%
\label{fig:ex3-d}%
\includegraphics[scale=0.145]{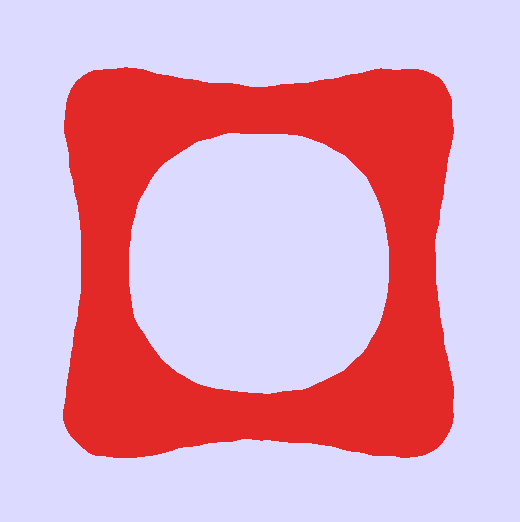}}%
\hspace{25pt}%
\subfigure[][]{%
\label{fig:ex3-e}%
\includegraphics[scale=0.145]{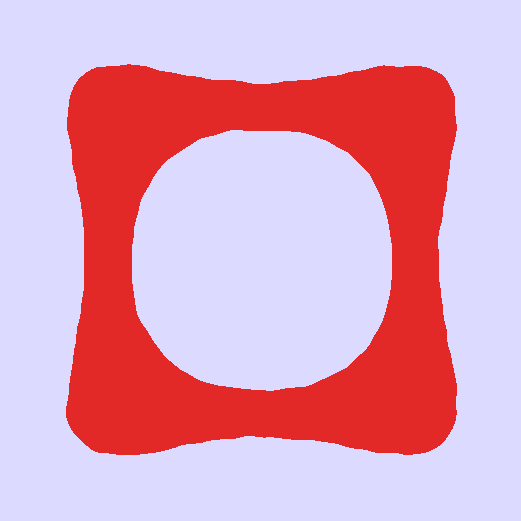}}%
\hspace{25pt}%
\subfigure[][]{%
\label{fig:ex3-f}%
\includegraphics[scale=0.145]{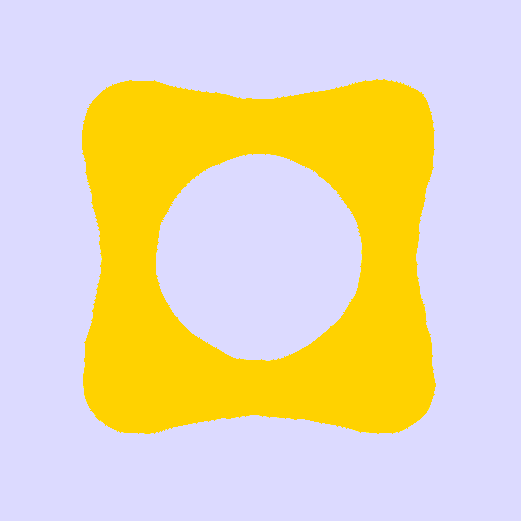}}%
\caption[Time optimal shape $T=5$ - Static shape]{Time optimal shape $T=5$ - Static shape:
\subref{fig:ex3-a} $t = 0$;
\subref{fig:ex3-b} $t = 0.5$;
\subref{fig:ex3-c} $t \in [1,4]$; 
\subref{fig:ex3-d} $t = 4.5$;
\subref{fig:ex3-e} $t = T$;
\subref{fig:ex3-f} static shape}%
\label{fig:ex3}%
\end{figure}

We plot in Figure \ref{fig:ex3} the comparison between the optimal shape at several times (in red) and the optimal static shape (in yellow). We see the same behavior when $t=\frac{T}{2}$.

Now in order to highlight the turnpike phenomenon, we plot the evolution in time of the distance between the optimal dynamic triple and the optimal static one $t \mapsto \Vert y_T(t)-\bar{y} \Vert+\Vert p_T(t)-\bar{p} \Vert+\Vert \chi_{\omega_T(t)}-\chi_{\bar{\omega}} \Vert$.
\begin{figure}[!h]
   \centering
   \includegraphics[scale=0.18]{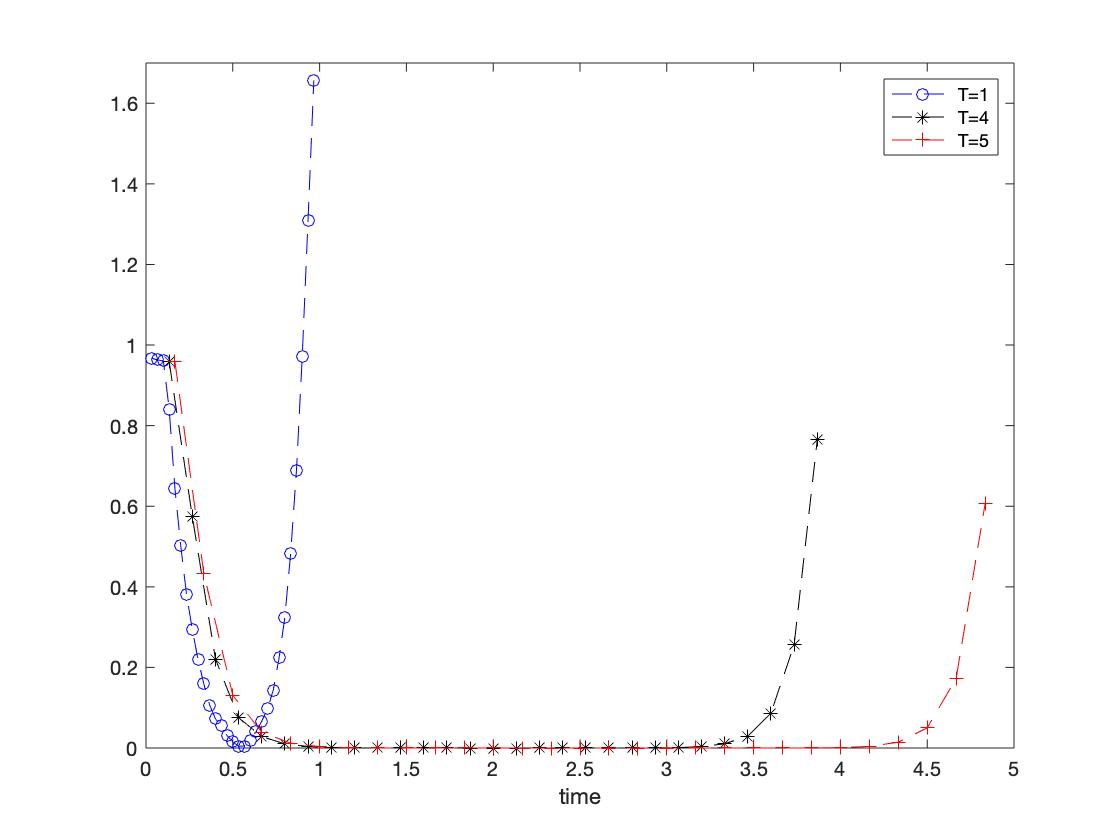}
   \caption{Error between dynamical optimal triple and static one}
   \label{expfig}
\end{figure}
In Figure \ref{expfig} we observe that this function is exponentially close to $0$. This behavior leads us to conjecture that the exponential turnpike property should be satisfied. 

To complete this work, we need to clarify the existence of optimal shapes for \osd \, when $y_d$ is convex. We see numerically in Figure \ref{fig:ex3} the time optimal shape's existence for $y_d$ convex on $\Omega$. Otherwise we can sometimes observe a relaxation phenomenon due to the presence of $\bar{c}$ and $c_T(\cdot)$ in the optimality conditions (\ref{OCocp}) - (\ref{OCsop}). 

We consider the same problem \ocp\, in 2D with $\Omega = [-1,1]^2$, $L = \frac{1}{8}, T = 5$ and the static one associated (\sop). We take $y_d(x,y) = -\frac{1}{20}(x^2+y^2-2)$. 

\begin{figure}[!h]%
\centering
\subfigure[][]{%
\label{fig:ex4-a}%
\includegraphics[scale=0.12]{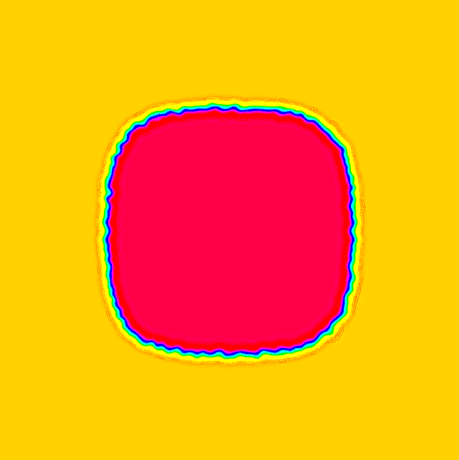}}%
\hspace{25pt}%
\subfigure[][]{%
\label{fig:ex4-b}%
\includegraphics[scale=0.12]{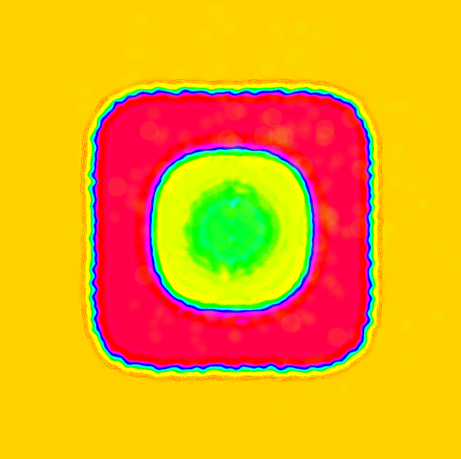}}%
\hspace{25pt}%
\subfigure[][]{%
\label{fig:ex4-c}%
\includegraphics[scale=0.12]{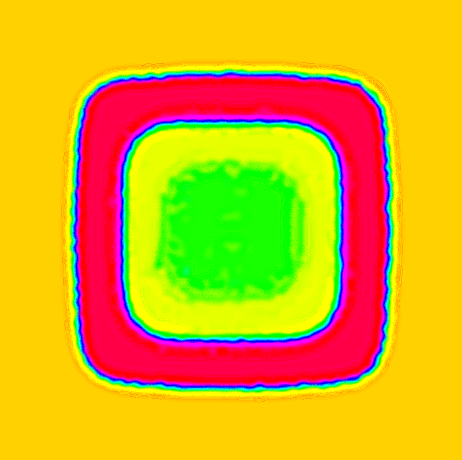}} \\\vspace{6pt}
\hspace{10pt}
\subfigure[][]{%
\label{fig:ex4-d}%
\includegraphics[scale=0.12]{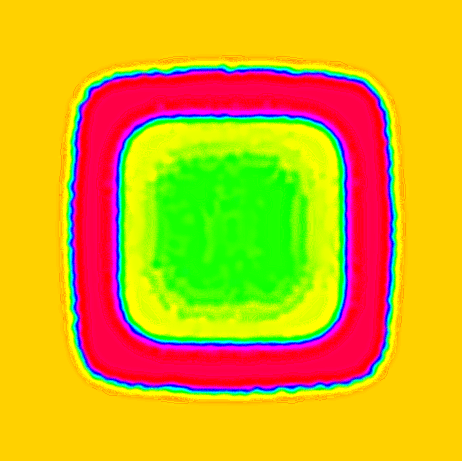}}%
\hspace{25pt}%
\subfigure[][]{%
\label{fig:ex4-e}%
\includegraphics[scale=0.12]{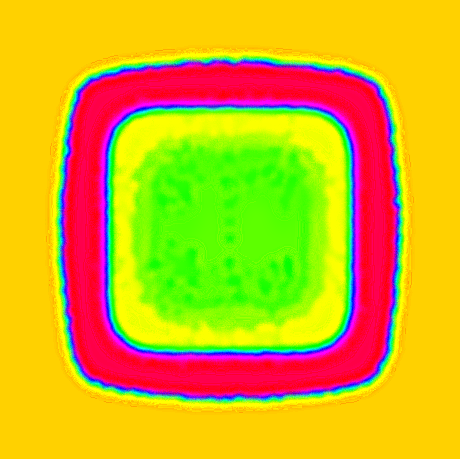}}%
\hspace{25pt}%
\subfigure[][]{%
\label{fig:ex4-f}%
\includegraphics[scale=0.12]{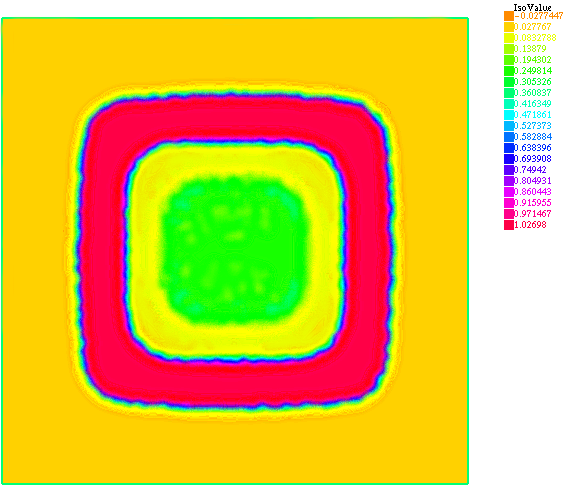}}%
\caption[Relax]{Relaxation phenomenon :
\subref{fig:ex4-a} $t = 0$;
\subref{fig:ex4-b} $t = 0.5$;
\subref{fig:ex4-c} $t \in [1,4]$; 
\subref{fig:ex4-d} $t = 4.5$;
\subref{fig:ex4-e} $t = T$;
\subref{fig:ex4-f} static shape}%
\label{fig:ex4}%
\end{figure}

\begin{figure}[!h]
   \centering
   \includegraphics[scale=0.18]{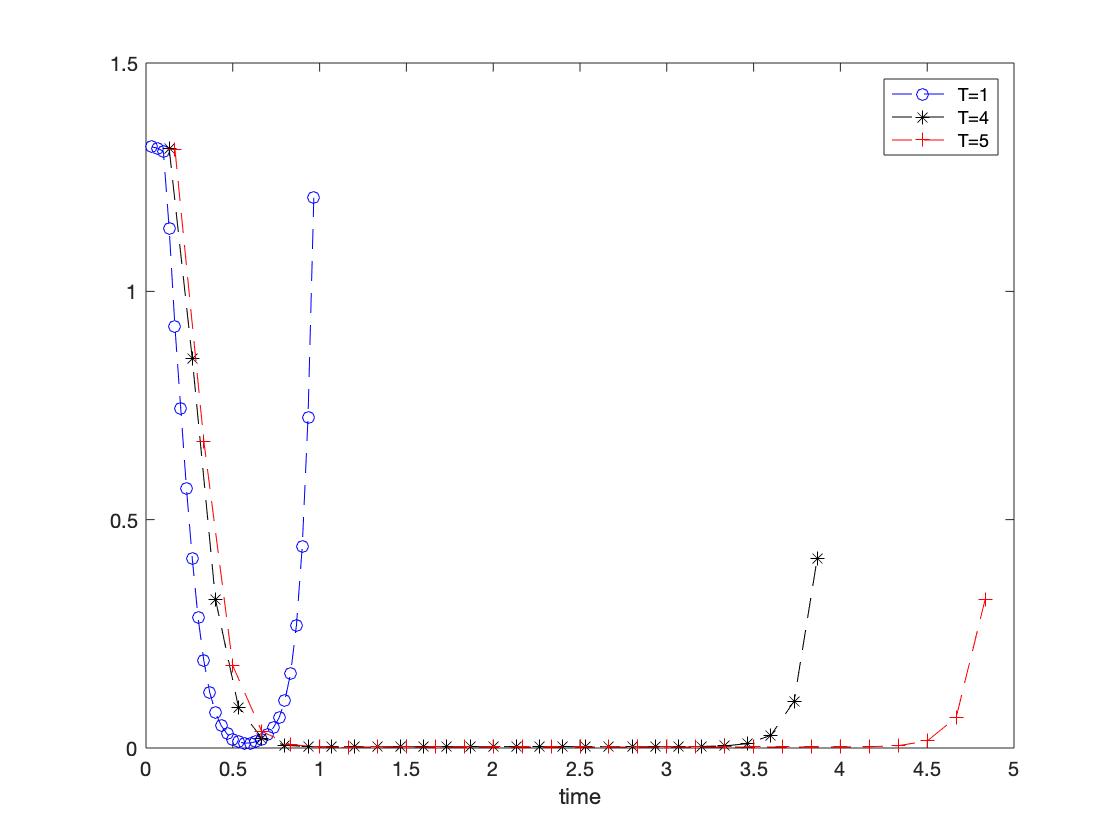}
   \caption{Error between dynamical optimal triple and static one (Relaxation case)}
   \label{expfig_relax}
\end{figure}

In Figure \ref{fig:ex4} we see that optimal control $(a_T,\bar{a})$ of \ocp\, and (\sop)\, take values in $(0,1)$ in the middle of $\Omega$. This illustrates that relaxation occurs for some $y_d$. Here, $y_d$ was chosen such that $-\triangle y_d \in (0,1)$. We have tuned the parameter $L$ to observe the relaxation phenomenon, but for same $y_d$ and smaller $L$, optimal solutions are shapes. Despite the relaxation we see in Figure \ref{expfig_relax} that turnpike still occurs.

\section{Further comments}
Numerical simulations when $\triangle y_d>0$ lead us to conjecture existence of an optimal shape for \osd, because we have not observed any relaxation phenomenon in that case. Existence might be proved thanks to arguments like maximal regularity properties and H\"older estimates for solutions of parabolic equations. 

Moreover, still based on our simulations and particularly on Figure \ref{expfig}, we conjecture the exponential turnpike property.

The work that we presented here is focused on second-order parabolic equations and particularly on the heat equation. Concerning the Mayer case, we have used in our arguments the Weyl law, sup-norm estimates of eigenelements (see \cite{MR3186367}) and analyticity of solutions (analytic-hypoelliptic operator). Nevertheless, concerning the Lagrange case and having in mind \cite{MR3616131,measureturnpikeTZ} it seems reasonable to extend our results to general local parabolic operators which satisfy an energy inequality \eqref{energy} and the maximum principle to ensure existence of solutions. However, some results like Theorem \ref{existencethm}.2-(ii) should be adapted. Moreover we consider a linear partial differential equation which gives uniqueness of the solution thanks to the strict convexity of the criterion. At the contrary, if we do not have uniqueness, as in \cite{measureturnpikeTZ}, the notion of measure-turnpike seems to be a good and soft way to obtain turnpike results.

To go further with the numerical simulations, our objective will be to find optimal shapes evolving in time, solving dynamical shape design problems for more difficult real-life partial differential equations which play a role in fluid mechanics for example. We can find in the recent literature some articles on the optimization of a wavemaker (see \cite{dalphinwave, doi:10.1093/imamat/hxu051}). It is natural to wonder what can happen when considering a wavemaker whose shape can evolve in time. We have in mind the behavior of a static wave that we can observe in the nature (Eisbach Wave in M\"unchen) which arises thanks to the shape of the bottom and when the inside flow is supercritical. We are interested in modeling this phenomenon and taking into account a bottom whose shape may evolve in time in order to design a static wave. Since the target is stationary, we would expect that an optimal evolving bottom stays most of the time static too. There already exist some wavemakers designed for surf-riding inspired by this phenomenon (see \cite{citywave}).

\appendix
\section{Energy inequality}\label{sec_app}
We recall some useful inequalities to study existence and turnpike. Since $\theta$ satisfies \eqref{ineq_theta}, we can find $\beta > 0, \gamma \geq 0$ such that $\beta \geq \gamma$ and
\begin{equation}
\label{ineqenergyellip}
(Au,u)\geq\beta\Vert u\Vert^2_{H^1_0(\Omega)} - \gamma\Vert u\Vert^2_{L^2(\Omega)}.
\end{equation}
From this follows the energy inequality (see \cite[Chapter 7, Theorem 2]{MR2597943}): there exists $C>0$ such that, for any solution $y$ of (\ref{heat_convex}), for almost every $t\in[0,T]$,
\begin{equation}
\Vert y(t) \Vert^{2} + \int_{0}^{t} \Vert y(s) \Vert_{H_0^1(\Omega)}^{2} \,ds\leq C\left(\Vert y_{0} \Vert^{2} + \int_{0}^{t}\Vert  a(s) \Vert^{2} \, ds\right).
\label{energy}
\end{equation}
We improve this inequality. Having in mind \eqref{ineqenergyellip}, the Poincar\'e inequality and that $y$ verifies \eqref{heat_convex}, we find two constants $C_1,C_2>0$ such that $\frac{d}{dt}\Vert y(t) \Vert^2 + C_1 \Vert y(t) \Vert^2 = f(t) \leq C_2 \Vert a(t) \Vert^2$. We solve this differential equation to get $\Vert y(t) \Vert^2 = \Vert y_0 \Vert^2 e^{-C_1 t} + \int_0^t e^{-C_1(t-s)} f(s)\,ds$. Since for all $t\in(0,T), f(t) \leq C_2 \Vert a(t) \Vert^2$, we obtain that for almost every $t\in (0,T)$,
\begin{equation}
\Vert y(t) \Vert^{2} \leq\ \Vert y_{0} \Vert^{2}e^{-C_1 t} + C_2 \int_{0}^{t}e^{-C_1(t-s)}\Vert a(s) \Vert^{2} \, ds.
\label{gronwall}
\end{equation}
The constants $C,C_1,C_2$ depend only on the domain $\Omega$ (Poincar\'e inequality) and on the operator $A$ and not on final time $T$ since \eqref{ineqenergyellip} is satisfied with $\beta \geq \gamma$. 

\paragraph*{\bf Acknowledgment}
{\small This project has received funding from the Grants ICON-ANR-16-ACHN-0014 and Finite4SoS ANR-15-CE23-0007-01 of the French ANR, the European Research Council (ERC) under the European Union's Horizon 2020 research and innovation programme (grant agreement 694126-DyCon), the Alexander von Humboldt-Professorship program, the Air Force Office of Scientific Research under Award NO: FA9550-18-1-0242, Grant MTM2017-92996-C2-1-R COSNET of MINECO (Spain) and by the ELKARTEK project KK-2018/00083 ROAD2DC of the Basque Government, Transregio 154 Project *Mathematical Modelling, Simulation and Optimization using the Example of Gas Networks* of the German DFG, the European Union Horizon 2020 research and innovation programme under the Marie Sklodowska-Curie grant agreement 765579-ConFlex.}

\bibliographystyle{abbrv}
\bibliography{bibliography.bib}

\end{document}